\documentclass[11pt]{article}

\usepackage[margin=1in]{geometry} 
\usepackage{amsmath,amsthm,amssymb,amsfonts}
\usepackage{graphicx, multicol, array}
\usepackage{cases}
\usepackage{enumerate}
\usepackage{enumitem}
\usepackage{hyperref}
\usepackage{xcolor}

\usepackage{mathrsfs}

\DeclareMathOperator{\ord}{ord}

\DeclareMathOperator{\rk}{rk}

\newtheorem{thm}{Theorem}[section]
\newtheorem{lem}{Lemma}[section]
\newtheorem{conj}{Conjecture}[section]
\newtheorem{exa}{Example}[section]
\newtheorem{cor}{Corollary}[section]
\newtheorem{dfn}{Definition}[section]

\title{\textbf{Note On Elliptic Groups of Prime Orders}}
\date{}
\author{N. A. Carella}

\usepackage{fancyhdr}
\begin{document}
\thispagestyle{empty}
\maketitle

\textbf{\textit{Abstract}:} Let \(E\) be an elliptic curve of rank $\rk(E) \geq 0$, and let $E(\mathbb{F}_p)$ be the elliptic group of order 
$\#E(\mathbb{F}_p)=n$. The number of primes $p\leq x$ such that $n$ is prime is expected to be $\pi(x,E)=\delta(E)x/\log^2 x+o(x/\log^2 x)$, where $\delta(E)\geq 0$ is the density constant. This note proves a lower bound $\pi(x,E) \gg x/\log^2 x$. \let\thefootnote\relax\footnote{\date{} 
}  \\

\section{Introduction}\label{sec1}
The applications of elliptic curves in cryptography demands elliptic groups of certain orders $n=\#E(\mathbb{F}_p)$, and certain factorizations of the integers $n$. The extreme cases have groups of $\mathbb{F}_p$-rational points $E(\mathbb{F}_p)$ of prime orders, and small multiples of large primes. \\

\begin{conj}[Koblitz] \label{conj800.1}  
	Let $E:f(X,Y)=0$ be an elliptic curve  of of discriminant $\Delta \ne 0$ defined over the integers $\mathbb{Z}$ which is not $\mathbb{Q}$-isogenous to a curve with nontrivial $\mathbb{Q}$-torsion and does not have CM. Then, 
	\begin{eqnarray} \label{eq800.01}  
	\pi(x,E)&=& \#\{ p \leq x: p \not | \Delta \text{ and } \#E(\mathbb{F}_p)=\text{prime}\} \\ 
	&=&\frac{x}{\log^2x}\prod_{p\geq 2} \left (1 -\frac{p^2-p-1}{(p-1)^3(p+1)}\right )+O\left ( \frac{x}{\log^3x}\right) \nonumber,
	\end{eqnarray}  
for large $x\geq 1$.
\end{conj}
A new refined version of this conjecture, for any elliptic curve over any number field $\mathcal{K}$, was recently proposed. The simpler case for the rational number field is stated below.
\begin{conj}[\cite{ZD09}] \label{conj800.2}  Let $E$ be an elliptic curve defined over the rational number field $\mathbb{Q}$, and let $t\geq 1$ be a positive integer. Then, there is an explicit constant $C(E,t) > 0$ such that  
	\begin{eqnarray}
	\pi(x,E,t)&=& \#\{ p \leq x: p \not | \Delta \text{ and } \#E(\mathbb{F}_p)/t=\text{prime}\} \\ 
	&=&C(E,t) \frac{x}{\log^2x}+O\left ( \frac{x}{\log^3x}\right) \nonumber,
	\end{eqnarray}  
for large $x\geq 1$.
\end{conj}
The product expression appearing in the above formula (\ref{eq800.01}  ) is basically the average density of prime orders, some additional details are given in Section \ref{sec8}. A result for groups of prime orders generated by primitive points is proved here. Let
\begin{equation}
\pi(x,E)= \#\{ p \leq x: p \not | \Delta \text{ and } d_E^{-1}\cdot\#E(\mathbb{F}_p)=\text{prime}\}.
\end{equation} 
The parameter $t=d_E\geq 1$ is a small integer defined in Section \ref{sec77}.\\

\begin{thm} \label{thm800.1}   
	Let $E:f(X,Y)=0$ be an elliptic curve over the rational numbers $\mathbb{Q}$ of rank $\rk(E(\mathbb{Q})>0$. Then, as $x \to \infty$,
	\begin{equation}
	\pi(x,E)\geq \delta(d_E,E)\frac{x}{\log^2 x} \left (1+O\left ( \frac{x}{\log x} \right )\right) ,
	\end{equation}  	
	where $\delta(d_E,E)$ is the density constant.
\end{thm}

The proof of this result is split into several parts. The next sections are intermediate results. The proof of Theorem \ref{thm800.1} is assembled in the penultimate section, and the last section has examples of elliptic curves with infinitely many elliptic groups $E(\mathbb{F}_p)$ of prime orders $n$.\\ 

\section{Representation of the Characteristic Function} \label{sec2}
\subsection{Primitive Points Tests}
For a prime $p \geq 2$, the group of points on an elliptic curve $E:y^2=f(x)$ is denoted by $E(\mathbb{F}_p)$. Several definitions and elementary properties of elliptic curves and the $n$-division polynomial $\psi_n(x,y)$ are sketched in Chapter 14. \\

\begin{dfn}
The order $\min \{k \in \mathbb{N}: kP=\mathcal{O} \}$ of an elliptic point is denoted by $\ord_E(P)$. A point is a \textit{primitive point} if and only if $\ord_E(P)=n$. 
\end{dfn}

\begin{lem} \label{lem3.3b}
	If $ E(\mathbb{F}_p)$ is a cyclic group, then it contains a primitive point $P \in E(\mathbb{F}_p)$.  
\end{lem}

\begin{proof}
	By hypothesis, $ E(\mathbb{F}_p) \cong \mathbb{Z}/n\mathbb{Z}$, and the additive group $\mathbb{Z}/n\mathbb{Z}$ contains $\varphi(n)\geq 1$ generators (primitive roots). 
\end{proof}

More generally, there is map into a cyclic group
\begin{equation}
E(\mathbb{Q})/E(\mathbb{Q})_{\normalfont{tors}} \longrightarrow \mathbb{Z}/m\mathbb{Z},  
\end{equation}
for some $m=\#E(\mathbb{F}_p)/d$, with $d \geq 1$; and the same result stated in the Lemma holds in the smaller cyclic group $\mathbb{Z}/m\mathbb{Z}\subset\mathbb{Z}/n\mathbb{Z}$.\\

\begin{lem} \label{lem3.4}
	Let $\# E(\mathbb{F}_p)=n$ and let $P \in E(\mathbb{F}_p)$. Then, $P$ is a primitive point if and only if $(n/q)P\ne \mathcal{O}$ for all prime divisors $q\,|\,n$. 
\end{lem}

Basically, this is the classical Lucas-Lehmer primitive root test applied to the group of points $E(\mathbb{F}_p)$. Another primitive point test intrinsic to elliptic curves is the $n$-division polynomial test.\\

\begin{lem} \label{lem3.5} {\normalfont ($n$-Division primitive point test)}
	Let $\# E(\mathbb{F}_p)=n$ and let $P \in E(\mathbb{F}_p)$. Then, $P$ is a primitive point if and only if $\psi_{n/q}(P) \ne 0$ for all prime divisors $q\,|\,n$. 
\end{lem}

The basic proof stems from the division polynomial relation
\begin{equation}\label{80-99}
mP=\mathcal{O} \Longleftrightarrow \psi_{m}(P)= 0,
\end{equation}
see \cite[Proposition 1.25]{SZ03}. The elliptic primitive point test calculations in the penultimate lemma takes place in the set of integer pairs $\mathbb{Z} \times \mathbb{Z}$, while the calculations for the $n$-division polynomial primitive point test takes place over the set of integers $\mathbb{Z}$. The elementary properties of the $n$-division polynomials are discussed in \cite{SZ03}, and the periodic property of the $n$-division polynomials appears in \cite{SJ05}.\\

\subsection{Additive Elliptic Character}
The discrete logarithm function, with respect to the fixed primitive point $T$, maps the group of points into a cyclic group. The diagram below specifies the  basic assignments.

\begin{equation}
\begin{array} {cll}
E(\mathbb{F}_p)& \longrightarrow & \mathbb{Z}/n\mathbb{Z}, \\
\mathcal{O}& \longrightarrow &\log_T (\mathcal{O})=0, \\
T& \longrightarrow &\log_T(T)=1. \\
\end{array} 
\end{equation}

In view of these information, an important character on the group of $E(\mathbb{F}_p)$-rational points can be specified.

\begin{dfn}
	A nontrivial additive elliptic character $\chi \bmod n$ on the group $E(\mathbb{F}_p)$ is defined by
	\begin{equation} \label{200-09}
	\chi(\mathcal{O})=e^{\frac{i2 \pi}{n}\log_T\mathcal({O})}=1, 
	\end{equation}
	and
	\begin{equation} \label{200-10}
	\chi(mT)=e^{\frac{i2 \pi}{n}\log_T(mT)}=e^{i2 \pi m/n}, 
	\end{equation}
	where $\log_T(mT)=m$ with $m \in \mathbb{Z}$.
\end{dfn}

\subsection{Divisors Dependent Characteristic Function}
A characteristic function for primitive points on elliptic curve is described in the \cite[p.\ 5]{SV11}; it was used there to derive a primitive point search algorithm.

\begin{lem} \label{lem3.6}
	Let $E$ be a nonsingular elliptic curve, and let $E(\mathbb{F}_p)$ be its group of points of cardinality $n=\#E(\mathbb{F}_p)$. Let $\chi$ be the additive character of order $d$ modulo $d$, and assume $P=(x_0,y_0)$ is a point on the curve. Then, 
	\begin{equation}\label{el33007}
	\Psi_E (P)=\sum _{d \,|\, n} \frac{\mu (d)}{d}\sum _{0 \leq t <d} \chi(tP)=
	\left \{\begin{array}{ll}
	1 & \text{ if } \ord_E (P)=n,  \\
	0 & \text{ if } \ord_E (P)\neq n. \\
	\end{array} \right .
	\end{equation}
\end{lem}

\subsection{Divisors Free Characteristic Function}
A new \textit{divisors-free} representation of the characteristic function of elliptic primitive points is developed here. This representation can overcomes some of the limitations of its equivalent in \ref{lem3.6} in certain applications. The \textit{divisors representation} of the characteristic function of elliptic primitive points, Lemma \ref{lem3.6}, detects the order \(\text{ord}_E (P)\) of the point \(P\in E( \mathbb{F}_p) \) by means of the divisors of the order \(n=\#E( \mathbb{F}_p) \). In contrast, the \textit{divisors-free representation} of the characteristic function, Lemma \ref{lem3.7}, detects the order \(\text{ord}_E(P) \geq 1\) of a point \(P\in E(\mathbb{F}_p)\) by means of the solutions of the equation $mT-P=\mathcal{O}$, where \(P,T \in E(\mathbb{F}_p)\) are fixed points, \(\mathcal{O}\) is the identity point, and $m $ is a variable such that $0\leq m \leq n-1$, and $\gcd (m,n)=1$. \\

\begin{lem} \label{lem3.7}
	Let \(p\geq 2\) be a prime, and let \(T\) be a primitive point in \(E(\mathbb{F}_p)\). For a nonzero point \(P \in
	E(\mathbb{F}_p)\) of order $n$ the following hold:
	If \(\chi \neq 1\) is a nonprincipal additive elliptic character of order \(\ord \chi =n\), then
	\begin{equation}
	\Psi_E (P)=\sum _{\gcd (m,n)=1} \frac{1}{n}\sum _{0\leq r\leq n-1} \chi \left ((mT-P)r\right)
	=\left \{
	\begin{array}{ll}
	1 & \text{ if } \ord_E(P)=n,  \\
	0 & \text{ if } \ord_E(P)\neq n, \\
	\end{array} \right .
	\end{equation}
	where \(n=\#
	E(\mathbb{F}_p)\) is the order of the rational group of points.
\end{lem}

\begin{proof} As the index \(m\geq 1\) ranges over the integers relatively prime to \(n\), the element \(mT\in E(\mathbb{F}_p)\) ranges over the elliptic primitive points. Ergo, the linear equation 
	\begin{equation}
	mT-P=\mathcal{O},
	\end{equation}
	where \(P,T \in E(\mathbb{F}_p)\) are fixed points, \(\mathcal{O}\) is the identity point, and $m $ is a variable such that $0\leq m \leq n-1$, and $\gcd (m,n)=1$, has a solution if and only if the fixed point \(P\in E(\mathbb{F}_p)\) is an elliptic primitive point. Next, replace \(\chi (t)=e^{i 2\pi  t/n }\) to obtain
	\begin{equation}
	\Psi_E(P)=\sum_{\gcd (m,n)=1} \frac{1}{n}\sum_{0\leq r\leq n-1} e^{i 2\pi  \log_T(mT-P)r/n }=
	\left \{\begin{array}{ll}
	1 & \text{ if } \ord_E (P)=n,  \\
	0 & \text{ if } \ord_E (P)\neq n, \\
	\end{array} \right.
	\end{equation}
	
	This follows from the geometric series identity $\sum_{0\leq k\leq N-1} w^{ k }=(w^N-1)/(w-1)$ with $w \ne 1$, applied to the inner sum.   
\end{proof}

\section{Primes In Short Intervals} \label{sec3}

There are many unconditional results for the existence of primes in short intervals $[x,x+y]$ of subsquareroot length $y\leq x^{1/2}$ for almost all large numbers $x\geq 1$. One of the earliest appears to be the Selberg result for $y=x^{19/77}$ with $x\leq X$ and 
$O\left (X (\log X)^{2} \right )$ exceptions, see \cite[Theorem 4]{SA43}. Recent improvements based on different analytic methods are given in \cite{WN95}, \cite{HG07}, and other authors. One of these results, but not the best, has the following claim. This calculation involves the weighted prime indicator function (vonMangoldt), which is defined by
\begin{equation}\label{800-12}
\Lambda(n)=
\begin{cases}
\log n &
 \text{ if } n=p^k, k \geq 1,\\
0 &\text{ if } n \ne p^k, k \geq 1,
\end{cases}
\end{equation}
where $p^k,k\geq 1$, is a prime power.\\  
 
\begin{thm} (\cite{HG07}) \label{thm4.1}
Given $\varepsilon>0$, almost all intervals of the form $[x-x^{1/6+\varepsilon},x]$ contains primes except for a set of $x \in [X,2X]$ with measure $O\left (X (\log X)^{-C+1} \right )$, and $C>1$ constant. Moreover, 
\begin{equation}\label{key}
\sum_{x-y 	\leq n \leq x} \Lambda(n)>\frac{y}{2}.
\end{equation}
\end{thm}

\begin{proof}The proof is based on zero density theorems, see \cite[p.\ 190]{HG07}.
\end{proof}
An introduction to zero density theorems and its application to primes in short intervals are given in \cite[p. 98]{KA93}, \cite[p.\ 264]{IK04}, et cetera. \\

The same result works for larger intervals $[x-x^{\theta +\varepsilon},x]$, where $\theta \in (1/6,1/2)$ and $(1-\theta)C<2$, with exceptions $O\left (X (\log X)^{-C+1} \right ), C>1$.

\section{Evaluation Of The Main Term}
A lower bound for the main term $M(x)$ in the proof of Theorem \ref{thm800.1} is evaluated here. 

\begin{lem} \label{lem800.1}
	Let \(x\geq 1\) be a large number, and let $p \in [x,2x]$ be prime. Then, 
	\begin{equation} \label{800-100}
	\sum_{x \leq p \leq 2x} \frac{1}{4 \sqrt{p} }  \sum_{p-2 \sqrt{p}\leq n\leq p+2\sqrt{p} } \frac{\Lambda(n)}{\log   n}\cdot\frac{1}{n}\sum_{\gcd(m,n)=1}1\gg\frac{1}{2} \frac{x}{\log^2 x} \left (1+  O  \left ( \frac{x}{\log x}  \right ) \right ).
	\end{equation} 
\end{lem}

\begin{proof}  The main term can be rewritten in the form
	\begin{eqnarray} \label{800-110}
	M(x)&=&\sum_{x \leq p \leq 2x} \frac{1}{4 \sqrt{p} }  \sum_{p-2 \sqrt{p}\leq n\leq p+2\sqrt{p}} \frac{\Lambda(n)}{\log n} \cdot\frac{1}{n}\sum_{\gcd(m,n)=1}1 \nonumber\\
	&=&\sum_{x \leq p \leq 2x} \frac{1}{4 \sqrt{p} }  \sum_{p-2 \sqrt{p}\leq n\leq p+2\sqrt{p}} \frac{\Lambda(n)} {\log n}\cdot\frac{\varphi(n)}{n} .
	\end{eqnarray} 
The Euler phi function $\varphi(n)=\#\{1\leq m<n:\gcd(m,n)=1\} $, and it has the value $\varphi(n)/n=1-1/n$ at a prime argument $n$. Thus, the expression 
\begin{equation} \label{800-103}
	 \frac{\Lambda(n)}{\log   n}\frac{\varphi(n)}{n} = 
\begin{cases}
 \frac{\Lambda(n)}{\log   n}\left (1-\frac{1}{ n} \right ) & \mbox{if } n=p^m, m \geq 1,\\
0 &\mbox{if } n=p^m, m \geq 1,
\end{cases}
\end{equation} 
is supported on the set of prim powers $n=p^m, m \geq 1$. Consequently,	
\begin{eqnarray} \label{800-110}
	M(x)&=&\sum_{x \leq p \leq 2x} \frac{1}{4 \sqrt{p} }  \sum_{p-2 \sqrt{p}\leq n\leq p+2\sqrt{p}} \frac{\Lambda(n)}{\log n} \left (1-\frac{1}{  n} \right )\nonumber \\
	&\geq& \frac{1}{2} \sum_{x \leq p \leq 2x} \frac{1}{4 \sqrt{p} } \cdot \frac{1}{\log p} \sum_{p-2 \sqrt{p}\leq n\leq p+2\sqrt{p}} \Lambda(n),
	\end{eqnarray} 
since $1-1/n \geq 1/2$. Now observe that as the prime $p \in [x,2x]$ varies, the number of intervals $[p-2 \sqrt{p},p+2\sqrt{p}]$ is the same as the number of primes in the interval $[x,2x]$, namely, 
	\begin{equation}\label{800-140}
	\pi(2x)-\pi(x)=\frac{x}{\log x} \left (1+ O \left (\frac{1}{\log x}  \right )\right )>\frac{x}{2\log x} 	
	\end{equation}
	for large $x\geq1$. By Theorem \ref{thm4.1}, the finite sum over the short interval satisfies
	\begin{equation}\label{800-115}
	\sum_{p-2 \sqrt{p}\leq n\leq p+2\sqrt{p}} \Lambda(n)>\frac{4\sqrt{p}}{2},
	\end{equation} 
	with $O \left (x\log^{-C}x  \right )$ exceptions $p \in [x,2x]$, where $1<C<4$. Take $C>2$, then, the number of exceptions is small in comparison to the number of intervals $\pi(2x)-\pi(x) > x/2\log x$ for large $x \geq 1$. Hence, an application of this Theorem yields
	\begin{eqnarray} \label{800-120}
	M(x) &\geq& \frac{1}{2}\sum_{x \leq p \leq 2x} \frac{1}{4 \sqrt{p} } \cdot \frac{1}{\log p} \sum_{p-2 \sqrt{p}\leq n\leq p+2\sqrt{p}} \Lambda(n) \nonumber \\
	&\geq&\frac{1}{2}\sum_{x \leq p \leq 2x} \frac{1}{4 \sqrt{p} } \cdot \frac{1}{\log p} \cdot \left( \frac{4\sqrt{p}}{2} \right ) +O \left (\frac{x}{\log^C x}  \right )\nonumber \\
	&\geq&\frac{1}{4}\frac{1}{\log x}\sum_{x \leq p \leq 2x} 1 +O \left (\frac{x}{\log^C x}  \right ) \\
	&\geq&\frac{1}{4}\frac{1}{\log x}\cdot  \frac{x}{\log x} \left (1+ O \left (\frac{1}{\log x}  \right )\right )  \nonumber, 
\end{eqnarray} 
where all the errors terms are absorbed into one term. 
\end{proof}

\textbf{Remark 4.1.} The exceptional intervals $[p-2 \sqrt{p},p+2\sqrt{p}]$, with $p \in [x,2x]$, contain fewer primes, that is,
\begin{equation}\label{800-210}
\sum_{p-2 \sqrt{p}\leq n\leq p+2\sqrt{p}} \Lambda(n)=o(\sqrt{p}).
\end{equation}
This shortfalls is accounted for in the correction term 
\begin{eqnarray}\label{800-213}
C(x)&=&\sum_{x \leq p \leq 2x} \frac{1}{4 \sqrt{p} } \cdot \frac{1}{\log p} \left (1-\frac{1}{  p} \right )\sum_{p-2 \sqrt{p}\leq n\leq p+2\sqrt{p}} \Lambda(n) \nonumber \\
&=&o \left(\sum_{x \leq p \leq 2x} \frac{1}{4 \sqrt{p} } \cdot \frac{1}{\log^2 p} \cdot \sqrt{p} \right ) \nonumber \\ &=&O \left (\frac{x}{\log^C x}  \right ),
\end{eqnarray}
where $C>2$, in the previous calculation.

\section{Estimate For The Error Term}
The analysis of an upper bound for the error term $E(x)$, which occurs in the proof of Theorem \ref{thm800.1}, is split into two parts. The first part in Lemma \ref{lem800.2} is an estimate for the triple inner sum. And the final upper bound is assembled in Lemma \ref{lem800.3}. \\

\begin{lem} \label{lem800.2}
	Let $E$ be a nonsingular elliptic curve over rational number, let $P \in E(\mathbb{Q})$ be a point of infinite order. Let \(x\geq 1\) be a large number. For each prime $p\geq 3$, fix a primitive point $T$, and suppose that $P \in E(\mathbb{F}_p)$ is not a primitive point for all primes $p\geq 2$, then
	
	\begin{equation} \label{800-220}	 
	\sum_{p-2 \sqrt{p}\leq n\leq p+2\sqrt{p} }\frac{\Lambda(n)}{\log n} \cdot \frac{1}{n }\sum_{\gcd(m,n)=1,}  
	\sum_{ 1 \leq r <n} \chi((mT-P)r) \leq 2 .	\end{equation} 
\end{lem}

\begin{proof} Let $\log_T:E(\mathbb{F}_p) \longrightarrow \mathbb{Z}_n$ be the discrete logarithm function with respect to the fixed primitive point $T$, defined by $\log_T(mT)=m$, $\log_T(P)=k$, and $\log_T(\mathcal{O})=0$. Then, the nontrivial additive character evaluates to
	\begin{equation}
	\chi(rmT)=e^{\frac{i2 \pi}{n}\log_T(rmT)}=e^{i2 \pi rm/n},
	\end{equation} and 
\begin{equation}
\chi(-rP)=e^{\frac{i2 \pi}{n}\log_T(-rP)}=e^{-i2 \pi rk/n},
\end{equation}
	respectively. To derive a sharp upper bound, rearrange the inner double sum as a product 
	\begin{eqnarray} \label{800-230}
	&& T(p) \nonumber \\&=&\sum_{p-2 \sqrt{p}\leq n\leq p+2\sqrt{p}}\frac{\Lambda(n)}{ \log n} \cdot \frac{1}{n }\sum_{\gcd(m,n)=1,}  
	\sum_{ 1 \leq r <n} \chi((mT-P)r) \nonumber \\
	&= &\sum_{p-2 \sqrt{p}\leq n\leq p+2\sqrt{p}} \frac{\Lambda(n)}{ \log n} \cdot \frac{1}{n }\sum_{ 1 \leq r <n} \chi(-rP)\sum_{\gcd(m,n)=1} \chi(rmT)\\
	&= &\sum_{p-2 \sqrt{p}\leq n\leq p+2\sqrt{p}}\frac{\Lambda(n)}{ \log n} \cdot \frac{1}{n }\left (\sum_{ 1 \leq r <n} e^{-i2 \pi rk/n} \right ) \left (\sum_{\gcd(m,n)=1} e^{i2 \pi rm/n} \right ) \nonumber.
	\end{eqnarray}
The hypothesis $mT-P\ne \mathcal{O}$ for $m\geq 1$ such that $\gcd(m,n)=1$ implies that the two inner sums in (\ref{800-230}) are complete geometric series, except for the terms for $m=0$ and $r=0$ respectively. Moreover, since $n\geq 2$ is a prime, the two geometric sums have the exact evaluations
	\begin{equation} \label{800-240}
	\sum_{ 0<r\leq n-1} e^{-i2 \pi rk/n}=-1    \qquad \text{ and } \qquad \sum_{\gcd(m,n)=1} e^{i2 \pi rm/n}=-1
	\end{equation} 
	for any $1 \leq k <n$, and $1\leq r < n$ respectively. Therefore, it reduces to
\begin{eqnarray} \label{800-233}
T(p)&=&\sum_{p-2 \sqrt{p}\leq n\leq p+2\sqrt{p}}\frac{\Lambda(n)}{ \log n}  \cdot \frac{1}{n }(-1)(-1)\nonumber \\
&\leq & \frac{1}{\log p}\sum_{ p-2 \sqrt{p}\leq n\leq p+2\sqrt{p}}\frac{\Lambda(n)}{n } \\
&\leq & \frac{1}{ \log p}\sum_{ n\leq p+2\sqrt{p}}\frac{\Lambda(n)}{n} \nonumber\\
&\leq  & 2 \nonumber,
\end{eqnarray}
refer to \cite[Theorem 2.7]{MV07} for additional details. 
\end{proof}

\begin{lem} \label{lem800.3}
	Let $E$ be a nonsingular elliptic curve over rational number, let $P \in E(\mathbb{Q})$ be a point of infinite order. For each large prime $p\geq 3$, fix a primitive point $T$, and suppose that $P \in E(\mathbb{F}_p)$ is not a primitive point for all primes $p\geq 2$, then
	
	\begin{equation} \label{800-400}
	\sum_{x\leq p\leq 2x,}\sum_{p-2 \sqrt{p}\leq n\leq p+2\sqrt{p} }   \frac{\Lambda(n)}{ \log n} \cdot \frac{1}{n }\sum_{\gcd(m,n)=1,} 
	\sum_{ 1 \leq r <n} \chi((mT-P)r) =O( x^{1/2} ).
	\end{equation} 
\end{lem}

\begin{proof} By assumption $P \in E(\mathbb{F}_p)$ is not a primitive point. Hence, the linear equation $mT-P= \mathcal{O}$ has no solution $m \in \{m:\gcd(m,n)=1\}$. This implies that the discrete logarithm $\log_T(mT-P)\ne0$, and $ \sum_{ 1 \leq r <n} \chi((mT-P)r) =-1$. This in turns yields   
	\begin{eqnarray} \label{800-410}
	&&|E(x)| \nonumber \\ &=&\left |\sum_{x \leq p \leq 2x } \frac{1}{4 \sqrt{p} } \sum_{p-2 \sqrt{p}\leq n\leq p+2\sqrt{p}}\frac{\Lambda(n)}{\log n} \cdot \frac{1}{n }\sum_{\gcd(m,n)=1,}  
	\sum_{ 1 \leq r <n} \chi((mT-P)r) \right |\nonumber \\
	&\leq &\sum_{x \leq p \leq 2x } \frac{1}{4 \sqrt{p} } \sum_{p-2 \sqrt{p}\leq n\leq p+2\sqrt{p} }\frac{\Lambda(n)}{ \log n}\cdot \frac{1}{n } \sum_{\gcd(m,n)=1} 1 \\
	&\ll &\sum_{x \leq p \leq 2x } \frac{1}{4 \sqrt{p} } \sum_{p-2 \sqrt{p}\leq n\leq p+2\sqrt{p}}\frac{\Lambda(n)}{2 \log n} \nonumber \\
	&\ll& \frac{x}{\log^2 x}+O\left (\frac{x}{\log^3 x} \right ) \nonumber.
	\end{eqnarray}
Here, the inequality $(1/n)\sum_{\gcd(m,n)=1}1=\varphi(n)/n \leq 1/2$ for all integers $n\geq 1$ was used in the second line. Hence, there is a sharper nontrivial upper bound for the error term. To derive a sharper upper bound, take absolute value, and apply Lemma \ref{lem800.2} to the inner triple sum to obtain this:
\begin{eqnarray} \label{800-420}
&&|E(x)| \nonumber \\
&\leq&\sum_{x \leq p \leq 2x } \frac{1}{4 \sqrt{p} } \left |  \sum_{p-2 \sqrt{p}\leq n\leq p+2\sqrt{p} }\frac{\Lambda(n)}{\log n}\cdot \frac{1}{n } \sum_{\gcd(m,n)=1,}  
\sum_{ 1 \leq r <n} \chi((mT-P)r) \right |\nonumber \\
&\leq &\sum_{x \leq p \leq 2x } \frac{1}{4 \sqrt{p} } \left (2  \right ) \nonumber \\
&\leq &  \frac{1}{\sqrt{x} }\sum_{x \leq p \leq 2x }1 \\
&=&O\left (x^{1/2} \right ) \nonumber.
\end{eqnarray}
The last line uses the trivial estimate $\sum_{x \leq p \leq 2x }1\leq x$.
    \end{proof}

\section{Elliptic Divisors} \label{sec77}
The divisor $\text{div}(f)=\gcd(f(\mathbb{Z}))$ of a polynomial $f(x)$ is the greatest common divisor of all its values over the integers, confer \cite[p.\ 395]{FI10}. Basically, the same concept extents to the setting of elliptic groups of prime orders, but it is significantly more complex. 

\begin{dfn} Let $\mathcal{O}_{\mathcal{K}}$ be the ring of integers of a quadratic numbers field $\mathcal{K}$. The elliptic divisor is an integer $d_E \geq1$ defined by
\begin{equation}\label{800-590}
d_E=\gcd \left (\{ \#E(\mathbb{F}_p): p\geq 2 \text{ and }p \text{ splits in } \mathcal{O}_{\mathcal{K}} \}  \right ).
\end{equation}
\end{dfn}
Considerable works, \cite{CA05}, \cite{MG15}, \cite{IJ08}, \cite{JJ08}, have gone into determining the elliptic divisors for certain classes of elliptic curves.\\
 
\begin{thm} \label{thm800-20} {\normalfont (\cite[Proposition 1]{JJ08})}
The divisor of an elliptic curve $E:y^2=x^3+ax+b$ over the rational numbers $\mathbb{Q}$ with complex multiplication by $\mathbb{Q}(\sqrt{D})$ and conductor $N$ satisfies $d_E |24$. The complete list, with $c,m \in \mathbb{Z}-\{0\}$, is the following.\\

\begin{tabular}{|c|c|c|c|c|c|}
\hline 
\rule[-1ex]{0pt}{2.5ex} $D$ & $(a,b)$ & $d_E$ & $D$ & $(a,b)$ & $d_E$ \\ 
	\hline 
	\rule[-1ex]{0pt}{2.5ex} $-3$ & $(0,m)$ & $1$ & $-7$ & $(-140c^2,-784c^3)$ & $4$ \\ 
	\hline 
	\rule[-1ex]{0pt}{2.5ex} $-3$ & $(0,m^2);(0,-27m^2)$ & $3$ & $-8$ & $(-30c^2,-56c^3)$ & $2$ \\ 
	\hline 
	\rule[-1ex]{0pt}{2.5ex} $-3$ & $(0,m^3)$ & $4$ & $-11$ &$(-1056c^2,-13552c^3)$  & $1$ \\ 
	\hline 
	\rule[-1ex]{0pt}{2.5ex} $-3$ & $(0,c^6);(0,27c^6)$ & $12$ & $-19$ &$(-608c^2,-5776c^3)$  & $1$ \\ 
	\hline 
	\rule[-1ex]{0pt}{2.5ex} $-4$ & $(m,0)$ & $2$ & $-43$ &$(-13760c^2,-621264c^3)$  & $1$ \\ 
	\hline 
	\rule[-1ex]{0pt}{2.5ex} $-4$ & $(m^2,0);(-m^2,0)$ & $4$ & $-67$ &$(-117920c^2,-15585808c^3)$  & $1$ \\ 
	\hline 
	\rule[-1ex]{0pt}{2.5ex} $-4$ & $(-c^4,0);(4c^4,0)$ & $8$ & $-163$ &$(-34790720c^2,-78984748304c^3)$  & $1$ \\ 
	\hline 
	\end{tabular}   
\end{thm}

\vskip .25 in

\section{Densities Expressions} \label{sec8}
The product expression appearing in Conjecture \ref{conj800.1}, id est,
\begin{equation}
P_0=\prod_{p\geq 2} \left (1 -\frac{p^2-p-1}{(p-1)^3(p+1)}\right ) \approx 0.505166168239435774,
\end{equation}  
is the basic the average density of prime orders, it was proved in \cite{KN88}, and very recently other proofs are given in \cite[Theorem 1 ]{BC11}, \cite{JN10}, \cite{LS14}, et alii. The actual density has a slight dependence on the elliptic curve $E$ and the point $P$. The determination of the dependence is classified into several cases depending on the torsion groups $E(\mathbb{Q})_{\text{tors}}$, and other parameters. \\

\begin{lem} \label{lem800.11} {\normalfont (\cite[Proposition 4.2]{ZD09})}. Let $E:f(x,y)=0$ be a Serre curve over the rational numbers. Let $D$ be the discriminant of the numbers field 
$\mathbb{Q}(\sqrt{\Delta})$, where $\Delta$ is the discriminant of any Weierstrass model of $E$ over $\mathbb{Q}$. If $d_E=1$, then

\begin{equation} 
\delta(1,E)=	
\begin{cases}
\displaystyle \left (1+ \prod_{q|D} \frac{1}{q^3-2q^2-q+3}\right )\prod_{p\geq 2} \left (1 -\frac{p^2-p-1}{(p-1)^3(p+1)}\right ) 
& \text{ if } D \equiv 1 \bmod 4;\\
\displaystyle \prod_{p\geq 2} \left (1 -\frac{p^2-p-1}{(p-1)^3(p+1)}\right ) & \text{ if } D \equiv 0 \bmod 4.
\end{cases}
\end{equation}
\end{lem}

\section{Elliptic Brun Constant}
Assuming the generalized Riemann hypothesis, the upper bound 
\begin{equation}\label{800-648}
\pi(x,E,t)\ll \frac{x}{\log^2 x}
\end{equation}
was proved in \cite{CA05} and \cite{ZD08}. In addition, the unconditional upper bound 
\begin{equation}\label{800-650}
\pi(x,E,t)\ll \frac{x}{(\log x)(\log \log \log x)}
\end{equation}
for elliptic curves with complex multiplication was proved in \cite[Proposition 7]{CA05}. Later, the same upper bound for elliptic curves without complex multiplication was proved in \cite[Theorem 1.3]{ZD08}. An improved version for any elliptic curve over the rational numbers is proved here.

\begin{lem}  \label{800-22} For any large number $x \geq 1$ and any elliptic curves $E:f(X,Y)=0$ of discriminant $\Delta\ne 0$,
	\begin{equation}\label{800-700}
	\pi(x,E) \leq \frac{12x}{\log^2x}.
	\end{equation}
\end{lem}

\begin{proof} The number of such elliptic primitive primes has the asymptotic formula
\begin{eqnarray} \label{800-860}
\pi(x,E)&=&\sum_{ \substack{p \leq x \\ \ord_E(P)=n \text{ prime}}} 1
\\
&=&\sum _{p\leq x} \frac{1}{4 \sqrt{p}}\sum _{p-2\sqrt{p}\leq n\leq p+2\sqrt{p}}\frac{\Lambda(n)}{\log n}\cdot  \Psi_E (P) \nonumber,
\end{eqnarray} 
where $\Psi_E(P)$ is the charateritic function of primitive points $P \in E(\mathbb{Q})$. This is obtained from the summation of the elliptic primitive primes density function over the interval $[1,x]$, see (\ref{800-500}).\\

Since $\Psi_E (P)=0,1$, the previous equation has the upper bound
\begin{eqnarray} \label{800-870}
\pi(x,E)
&\leq &\sum_{p\leq x} \frac{1}{4 \sqrt{p}}\sum _{p-2\sqrt{p}\leq n\leq p+2\sqrt{p}}\frac{\Lambda(n)}{\log n}\nonumber\\
&\leq &\sum_{p\leq x} \frac{1}{4 \sqrt{p}}\sum _{p-2\sqrt{p}\leq q\leq p+2\sqrt{p}}1\nonumber,
\end{eqnarray} 
where $q$ ranges over the primes in the short interval $[p-2\sqrt{p}, p+2\sqrt{p}]$. The inner sum is estimated using either the explicit formula or Brun-Titchmarsh theorem. The later result states that the number of primes $p$ in the short interval $[x,x+4\sqrt{x}]$ satisfies the inequality
\begin{equation} \label{800-40}
\pi(x+4\sqrt{x})-\pi(x) \leq \frac{3 \cdot 4\sqrt{x}}{ \log x},
\end{equation}
see \cite[p.\  167]{IK04}, \cite[Theorem 3.9]{MV07}, and \cite[p.\  83]{TG15}, and similar references. Replacing (\ref{800-40}) into (\ref{800-870}) yields

\begin{eqnarray} \label{800-880}
\sum_{p\leq x} \frac{1}{4 \sqrt{p}}\sum _{p-2\sqrt{p}\leq q\leq p+2\sqrt{p}}1 &\leq & \sum_{p\leq x} \frac{1}{4 \sqrt{p}} \left ( \frac{3 \cdot 4\sqrt{p}}{\log p} \right )\\
&\leq & 3\sum_{p\leq x} \frac{1}{\log p}\nonumber\\ 
&\leq&12 \frac{x}{\log^2 x}\nonumber.
\end{eqnarray} 
The last inequality follows by partial summation and the prime number theorem $\pi(x)= x/\log x+O(x/\log^2 x)$.
\end{proof}

This result facilitates the calculations of a new collection of constants associated with elliptic curves. The best known results have established the sum 
\begin{equation}\label{800-79}
\sum_{\substack{p \leq x\\ \#E(\mathbb{F}_p)=\text{prime}}}p^{-1}\ll \log \log \log x,
\end{equation}
but does not establish the convergence of this series, see \cite{CA05} and \cite{ZD08}. The convergence of this series is proved below.

\begin{cor} For any elliptic curve $E$, the elliptic Brun constant
	\begin{equation}\label{800-760}
	\sum_{\substack{p \geq 2\\ \#E(\mathbb{F}_p)=\text{prime}}} \frac{1}{p}< \infty
	\end{equation}
	converges.
\end{cor}

\begin{proof} Use the prime counting measure $\pi(x,E,t)\leq 6 x/\log^2 x+O(x/\log^3 x)$ in Theorem \ref{thm800.1} to evaluate the infinite sum
\begin{eqnarray}\label{800-333}
\sum_{\substack{p \geq 2\\ \#E(\mathbb{F}_p)=\text{prime}}} \frac{1}{p} &=& \int_2^{\infty}\frac{1}{z} d \pi(z,E,t) \nonumber \\
&=&O(1)+\int_2^{\infty}\frac{\pi(z,E,t)}{z^2} d z \\
&<& \infty \nonumber 
\end{eqnarray}
as claimed.
\end{proof} 

\begin{cor} For any elliptic curve $E$, the elliptic Brun constant
	\begin{equation}\label{800-760}
	\sum_{\substack{p \geq 2\\ \#E(\mathbb{F}_p)=\text{prime}}} \frac{1}{p}< \infty
	\end{equation}
	converges.
\end{cor}

\begin{proof} Use the prime counting measure $\pi(x,E)\leq 6 x/\log^2 x+O(x/\log^3 x)$ in Lemma \ref{800-22} to evaluate the infinite sum
\begin{eqnarray}\label{800-333}
\sum_{\substack{p \geq 2\\ \#E(\mathbb{F}_p)=\text{prime}}} \frac{1}{p} &=& \int_2^{\infty}\frac{1}{t} d \pi(t,E) \nonumber \\
&=&O(1)+\int_2^{\infty}\frac{\pi(t,E)}{t^2} d t \\
&<& \infty \nonumber 
\end{eqnarray}
as claimed.
\end{proof} 

The elliptic Brun constants and the numerical data for the prime orders of a few elliptic curves were compiled. The last example shows the highest density of prime orders. Accorddingly, it has the largest constant.\\   

\begin{exa}  { \normalfont  The nonsingular Bachet elliptic curve $E: y^2=x^3+2$ over the rational numbers has complex multiplication by $\mathbb{Z}[\rho]$, and nonzero rank $\rk(E)=1$. The data for $p \leq 1000$ with prime orders $n$ are listed on the Table \ref{t801}; 
and the elliptic Brun constant is
\begin{equation}\label{800-41}
\sum_{\substack{p \geq 2\\ \#E(\mathbb{F}_p)=\text{prime}}} \frac{1}{p}=.520067922 \ldots.
\end{equation}

\begin{table} \label{t801}
\begin{center}
\begin{tabular}{||c|c|c|c|c|c|c|c|c|c|c||}
	\hline 
	\rule[-1ex]{0pt}{2.5ex} $p$ & 3 & 13 & 19 & 61 & 67 & 73 & 139 & 163 & 211 & 331 \\ 
	\hline 
	\rule[-1ex]{0pt}{2.5ex} $n$ & 3 & 19 & 13 & 61 & 73 & 81 & 163 & 139 & 199 & 331 \\ 
	\hline 
	\rule[-1ex]{0pt}{2.5ex} $p$ & 349 & 541 & 547 & 571 & 613 & 661 & 757 & 829 & 877 &  \\ 
	\hline 
	\rule[-1ex]{0pt}{2.5ex} $n$ & 313 & 571 & 571 & 541 & 661 & 613 & 787 & 823 & 937 &  \\ 
	\hline 
\end{tabular} 
\end{center}
\caption{\label{859} Prime Orders $n=\#E(\mathbb{F}_p)$  modulo $p$ for $y^2=x^3+2$.}
\end{table}	
}
\end{exa}

\begin{exa}  { \normalfont  The nonsingular elliptic curve $E: y^2=x^3+ 6x-2$ over the rational numbers has no complex multiplication and zero rank $\rk(E)=0$. The data for $p \leq 1000$ with prime orders $n$ are listed on the Table \ref{t804}; 
and the elliptic Brun constant is
\begin{equation}\label{800-72}
\sum_{\substack{p \geq 2\\ \#E(\mathbb{F}_p)=\text{prime}}} \frac{1}{p}=.186641187 \ldots.
\end{equation}

\begin{table} \label{t804}
	\begin{center}
		\begin{tabular}{||c|c|c|c|c|c|c|c|c|c|c||}
			\hline 
			\rule[-1ex]{0pt}{2.5ex} $p$ & 3 & 7 & 97 & 103 & 181 & 271 & 313 & 367 & 409 & 487 \\ 
			\hline 
			\rule[-1ex]{0pt}{2.5ex} $n$ & 3 & 7 & 97 & 107 & 163 & 293 & 331 & 383& 397 & 499 \\ 
			\hline 
			\rule[-1ex]{0pt}{2.5ex} $p$ & 883 &  967&  & &  &  &  &  &  &  \\ 
			\hline 
			\rule[-1ex]{0pt}{2.5ex} $n$ & 853 & 941 &  & &  &  & &  &  &  \\ 
			\hline 
		\end{tabular} 
	\end{center}
	\caption{\label{869} Prime Orders $n=\#E(\mathbb{F}_p)$  modulo $p$ for $y^2=x^3+6x-2$.}
\end{table}	
}
\end{exa}

\begin{exa}   { \normalfont  The nonsingular elliptic curve $E: y^2=x^3-x$ over the rational numbers has complex multiplication by $\mathbb{Z}[\rho]$, and nonzero rank $\rk(E)=0$. The data for $p \leq 1000$ with prime orders $n/4$ are listed on the Table \ref{t806}; 
	and the elliptic Brun constant is
	\begin{equation}\label{800-73}
	\sum_{\substack{p \geq 2\\ \#E(\mathbb{F}_p)/4=\text{prime}}} \frac{1}{p}=.549568584 \ldots.
	\end{equation}
	
	\begin{table} \label{t806}
		\begin{center}
\begin{tabular}{||c|c|c|c|c|c|c|c|c|c|c||}
	\hline 
	\rule[-1ex]{0pt}{2.5ex} $p$ & 5 & 7 & 11 & 19 & 43 & 67 & 163 & 211 & 283 & 331 \\ 
	\hline 
	\rule[-1ex]{0pt}{2.5ex} $n$ & 2 & 2 & 3 & 5 & 11 & 17 & 41 & 53& 71 & 83 \\ 
	\hline 
	\rule[-1ex]{0pt}{2.5ex} $p$ & 523 & 547&691  &787 &907  &  &  &  &  &  \\ 
	\hline 
	\rule[-1ex]{0pt}{2.5ex} $n$ & 131 & 137 &173  &197 &227  &  & &  &  &  \\ 
	\hline 
\end{tabular}  
		\end{center}
		\caption{Prime Orders $n/4=\#E(\mathbb{F}_p)/4$  modulo $p$ for $y^2=x^3-x$.}
	\end{table}	
}	
\end{exa}

\begin{exa}   { \normalfont  The nonsingular elliptic curve $E: y^2=x^3-x$ over the rational numbers has complex multiplication by $\mathbb{Z}[\rho]$, and nonzero rank $\rk(E)=0$. The data for $p \leq 1000$ with prime orders $n/8$ are listed on the Table \ref{t808}; 
	and the elliptic Brun constant is
	\begin{equation}\label{800-77}
	\sum_{\substack{p \geq 2\\ \#E(\mathbb{F}_p)/8=\text{prime}}} \frac{1}{p}=.2067391731 \ldots.
	\end{equation}
	
	\begin{table} \label{t808}
		\begin{center}
			\begin{tabular}{||c|c|c|c|c|c|c|c|c|c|c||}
				\hline 
				\rule[-1ex]{0pt}{2.5ex} $p$&17 & 23 &29 &37 & 53 & 101 & 103 &109&149 &151  \\ 
				\hline 
				\rule[-1ex]{0pt}{2.5ex} $n$ & 2 & 3 & 5 & 5 & 5 & 13 & 13 & 13 & 17 & 19 \\ 
				\hline 
				\rule[-1ex]{0pt}{2.5ex} $p$ &157 &277  &293  &317  &389  &487 &541  &631  &661  &701  \\ 
				\hline 
				\rule[-1ex]{0pt}{2.5ex} $n$ &17  &37  &37  &41  &37  &53  &61  &73  &79  &89  \\ 			\hline 
			\rule[-1ex]{0pt}{2.5ex} $p$ &757 & 773 &797  &821  &823  &829 &853  &  &  &  \\ 
			\hline 
			\rule[-1ex]{0pt}{2.5ex} $n$ & 97 &101  &97  &109&103  &97  &101 &  &  &  \\ 
			 \hline 
				 \hline 
			\end{tabular} 
		\end{center}
		\caption{Prime Orders $n/8=\#E(\mathbb{F}_p)/8$  modulo $p$ for $y^2=x^3-x$.}
	\end{table}	
}	
\end{exa}

\begin{exa}   { \normalfont The nonsingular Bachet elliptic curve $E: y^2=x^3+1$ over the rational numbers has complex multiplication by $\mathbb{Z}[\rho]$, and nonzero rank $\rk(E)=?1$. The data for $p \leq 1000$ with prime orders $n/12$ are listed on the Table \ref{t820}; 
	and the elliptic Brun constant is
	\begin{equation}\label{800-71}
	\sum_{\substack{p \geq 2\\ \#E(\mathbb{F}_p)/12=\text{prime}}} \frac{1}{p}=.5495685884 \ldots.
	\end{equation}
	
	\begin{table} \label{t820}
		\begin{center}
			\begin{tabular}{||c|c|c|c|c|c|c|c|c|c|c||}
				\hline 
				\rule[-1ex]{0pt}{2.5ex} $p$ & 31 & 43 & 59 & 67 & 73 & 79 & 97 & 103 &131 & 139 \\ 
				\hline 
				\rule[-1ex]{0pt}{2.5ex} $n$ & 3 & 3 &5 & 7 & 7 & 7 & 7 &7 & 11 & 13 \\ 
				\hline 
				\rule[-1ex]{0pt}{2.5ex} $p$ & 151 & 163 & 181 & 199 & 227 & 241 & 337 & 367 & 379 &409  \\ 
				\hline 
				\rule[-1ex]{0pt}{2.5ex} $n$ & 13 &13 &13 & 19 & 19 & 19 & 31 & 31 & 31 &31  \\ 
				\hline 
				\rule[-1ex]{0pt}{2.5ex} $p$ & 421 &443 & 463 & 487 & 491 & 523 & 563 & 709 & 751 &787  \\ 
				\hline 
				\rule[-1ex]{0pt}{2.5ex} $n$ & 37 &37 &37 &37 &41 &43 & 47 & 61 & 67 &61  \\ 
				\hline 
				\rule[-1ex]{0pt}{2.5ex} $p$ & 823 & 829 & 859& 883 & 907 & 947 & 967 & 991 &  &  \\ 
				\hline 
				\rule[-1ex]{0pt}{2.5ex} $n$ & 73 &73 &67 & 73 & 79 & 79 & 79 & 79 &  &  \\ 
				\hline   
			\end{tabular} 
		\end{center}
		\caption{ Prime Orders $n/12=\#E(\mathbb{F}_p)/12$  modulo $p$ for $y^2=x^3+1$.}
	\end{table}	
}	
\end{exa}	

\newpage
\section{Prime Orders $n$} \label{sec7}
The characteristic function for primitive points in the group of points $E(\mathbb{F}_p)$ of an elliptic curve $E:f(X,Y)=0$ has the representation
\begin{equation} \label{800-29}
\Psi_E(P)=
\left \{\begin{array}{ll}
1 & \text{ if } \ord_E (P)=n,  \\
0 &  \text{ if } \ord_E (P) \ne n. \\
\end{array} \right.
\end{equation} 
The parameter $n=\# E(\mathbb{F}_p)$ is the size of the group of points, and the exact formula for $\Psi_E (P)$ is given in Lemma \ref{lem3.7}.\\ 

Since each order $n$ is unique, the weighted sum
\begin{equation} \label{800-500}
\frac{1}{4 \sqrt{p}}\sum _{p-2\sqrt{p}\leq n\leq p+2\sqrt{p}}\frac{\Lambda(n)}{\log n} \cdot \Psi_E (P) 
=\left \{\begin{array}{ll}
\displaystyle \frac{1}{4 \sqrt{p}} \cdot \frac{\Lambda(n)}{\log n} & \text{ if } \ord_E (P)=n \text{ and } n=q^k,  \\
0 &  \text{ if } \ord_E (P) \ne n \text{ or } n\ne q^k, \\
\end{array} \right.
\end{equation}
where $n=q^k,k\geq 1$, is a prime power, is a discrete measure for the density of elliptic primitive primes $p \geq 2$ such that $P \in E(\mathbb{\overline{Q}})$ is a primitive point of prime power order $n \in [p-2\sqrt{p}, p+2\sqrt{p}]$.\\

\begin{proof} \text{(Theorem \ref{thm800.1}).} Let $\langle P \rangle= E(\mathbb{F}_p)$ for at least one large prime $p \leq x_0$, and let $x \geq x_0 \geq 1$ be a large number. Suppose that \(P\not \in E(\mathbb{Q})_{\text{tors}} \) is not a primitive point of prime order $\ord_E(P)=n$ in $E(\mathbb{F}_p)$ for all primes \(p\geq x\). Then, the sum of the elliptic primes measure over the short interval \([x,2x]\) vanishes. Id est, 
\begin{equation} \label{800-510}
0=\sum _{x \leq p\leq 2x} \frac{1}{4 \sqrt{p}}\sum _{p-2\sqrt{p}\leq n\leq p+2\sqrt{p}} \frac{\Lambda(n)}{\log n} \cdot\Psi_E (P).
\end{equation}
Replacing the characteristic function, Lemma \ref{lem3.7}, and expanding the nonexistence equation (\ref{800-510}) yield
\begin{eqnarray} \label{800-520}
0&=&\sum _{x \leq p\leq 2x}  \frac{1}{4 \sqrt{p}}\sum _{p-2\sqrt{p}\leq  n\leq p+2\sqrt{p}}
 \frac{\Lambda(n)}{\log n} \cdot \Psi_E (P) \nonumber\\
&=&\sum _{x \leq p\leq 2x}  \frac{1}{4 \sqrt{p}} \sum _{p-2\sqrt{p}\leq  n\leq p+2\sqrt{p}}  \frac{\Lambda(n)}{\log n} \left (\frac{1}{n}\sum_{\gcd(m,n)=1} 
\sum_{ 0 \leq r \leq n-1} \chi ((mT-P)r ) \right ) \nonumber \\
&=&\delta(d_E,E)\sum _{x \leq p\leq 2x}  \frac{1}{4 \sqrt{p}} \sum _{p-2\sqrt{p}\leq  n\leq p+2\sqrt{p}} \frac{\Lambda(n)}{\log n}\cdot \frac{1}{n}\sum_{\gcd(m,n)=1}1
\\
& & \qquad+ \sum _{x \leq p\leq 2x}  \frac{1}{4 \sqrt{p}} \sum _{p-2\sqrt{p}\leq  n\leq p+2\sqrt{p}}  \frac{\Lambda(n)}{ \log n}\cdot \frac{1}{n}\sum_{\gcd(m,n)=1} 
\sum_{ 1 \leq r \leq n-1} \chi ((mT-P)r )\nonumber \\
&=&\delta(d_E,E)M(x) + E(x) \nonumber,
\end{eqnarray} 
where $\delta(d_E,E)\geq0$ is a constant depending on both the fixed elliptic curve $E:f(X,Y)=0$ and the integer divisor $d_E$. \\

The main term $M(x)$ is determined by a finite sum over the principal character \(\chi =1\), and the error term $E(x)$ is determined by a finite sum over the nontrivial multiplicative characters \(\chi \neq 1\).\\

Applying Lemma \ref{lem800.1} to the main term, and Lemma \ref{lem800.3} to the error term yield
\begin{eqnarray} \label{800-530}
\sum _{x \leq p\leq 2x} \frac{1}{4 \sqrt{p}}\sum _{p-2\sqrt{p}\leq n\leq p+2\sqrt{p}} \frac{\Lambda(n)}{\log n}\cdot \Psi_E (P)
&=&\delta(d_E,E)M(x) + E(x) \\
&\gg & \delta(d_E,E)\frac{x}{ \log^2 x} \left (1 +O \left (\frac{x}{\log x} \right ) \right ) \nonumber \\
&& \qquad \qquad +O\left (x^{1/2 }\right) \nonumber \\
&\gg&  \delta(d_E,E)\frac{x}{ \log^2 x} \left (1 +O \left (\frac{x}{\log x} \right ) \right ) \nonumber.
\end{eqnarray} 
But, if $\delta(d_E,E)>0$, the expression 
\begin{eqnarray} \label{800-540}
\sum _{x \leq p\leq 2x} \frac{1}{4 \sqrt{p}}\sum _{p-2\sqrt{p}\leq n\leq p+2\sqrt{p}} \frac{\Lambda(n)}{\log    n} \cdot \Psi_E (P)
&\gg&  \delta(d_E,E) \frac{x}{ \log^2 x} \left (1 +O \left (\frac{x}{\log x} \right ) \right ) \nonumber\\
&>&0,
\end{eqnarray} 
contradicts the hypothesis  (\ref{800-510}) for all large numbers $x \geq x_0$. Ergo, there are infinitely many primes $p\geq x $ such that a fixed elliptic curve of rank $\rk(E)>0$ with a primitive point $P$ of infinite order, for which the corresponding groups $E(\mathbb{F}_p)$ have prime orders. Lastly, the number of such elliptic primitive primes has the asymptotic formula
\begin{eqnarray} \label{800-560}
\pi(x,E)&=&\sum_{ \substack{p \leq x \\ \ord_E(P)=n \text{ prime}}} 1
 \nonumber \\
&=&\sum _{p\leq x} \frac{1}{4 \sqrt{p}}\sum _{p-2\sqrt{p}\leq n\leq p+2\sqrt{p}}\frac{\Lambda(n)}{\log n}\cdot  \Psi_E (P)    \\
&\geq&\delta(d_E,E) \frac{x}{ \log^2 x} \left (1+O \left (\frac{x}{\log x} \right )\right )\nonumber,
\end{eqnarray} 
which is obtained from the summation of the elliptic primitive primes density function over the interval $[1,x]$. 
\end{proof}

\section{Examples Of Elliptic Curves}
 The densities of several elliptic curves have been computed by several authors. Extensive calculations for some specific densities are given in \cite{ZD09}.\\

\begin{exa}   { \normalfont  The nonsingular Bachet elliptic curve $E: y^2=x^3+2$ over the rational numbers has complex multiplication by $\mathbb{Z}[\rho]$, and nonzero rank $\rk(E)=1$. It is listed as 1728.n4 in \cite{LMFDB}. The numerical data shows that $\# E(\mathbb{F}_p)=n$ is prime for at least one prime, see Table \ref{t801}. Hence, by Theorem \ref{thm800.1}, the corresponding group of $\mathbb{F}_p$-rational points $\# E(\mathbb{F}_p)$ has prime orders $n=\# E(\mathbb{F}_p)$ for infinitely many primes $p \geq 3$. \\

Since $\Delta=-2^6 \cdot 3^3$, the discriminant of the quadratic field $\mathbb{Q}(\sqrt{\Delta})$ is $D=-3$. Moreover, the integer divisor $d_E=1$ since $\# E(\mathbb{F}_p)$ is prime for at least one prime. Thus, applying Lemma \ref{lem800.11}, gives the natural density
\begin{equation}
\delta(1,E)=\frac{10}{9} P_0 \approx 0.5612957424882619712979385 \ldots,
\end{equation}
The predicted number of elliptic primes $p \nmid 6N$ such that $\# E(\mathbb{F}_p)$ is prime has the asymptotic 
formula
\begin{equation}
\pi(x,E)=\delta(1,E)\int_2^x \frac{1}{\log(t+1)} \frac{dt}{t}.
\end{equation}

A lower bound for the counting function is
\begin{equation}
\pi(x,E)\geq \delta(1,E) \frac{x}{\log^ 3 x} \left (1+ O\left (\frac{1}{\log x} \right ) \right ),
\end{equation}
see Theorem \ref{thm800.1}.\\

\begin{table}
\begin{center}
\begin{tabular}{||l|r c l||} 
		\hline
		\textbf{Invariant} & \textbf{Value}&&   \\ [1ex]  
		\hline\hline
		Discriminant  &$\Delta$&=&$-16(4a^3+27b^2)=-1728$\\ 
		\hline
		Conductor& $N$&=&$1728 $\\
		\hline
		j-Invariant &$j(E)$&=&$ (-48b)^3/\Delta=0$\\
		\hline
		Rank &$\rk(E)$&=& $1$ \\ 
		\hline
Special $L$-Value &$L^{
'}(E,1)$&$\approx$&$ 2.82785747365$\\
	\hline
		Regulator &$R$&=&$.754576$ \\ 
		\hline
		Real Period &$\Omega$&=&$5.24411510858$ \\ 
		\hline
		Torsion Group &$E(\mathbb{Q})_{\text{tors}}$&$\cong$&$\{ \mathcal{O} \}$ \\ 
		\hline
		Integral Points &$E(\mathbb{Z})$&=&$ \{ \mathcal{O} ,(-1,1);(-1,1)\}$ \\
		\hline
		Rational Group &$E(\mathbb{Q})$&$=$&$ \mathbb{Z}$ \\
		\hline
		Endomorphims Group &$End(E)$&=&$\mathbb{Z}[(1+\sqrt{-3})/2]$, CM \\
\hline
Integer Divisor  &$d_E$&=&$1$ \\
		\hline
\end{tabular}
\end{center}
\caption{\label{9600} Data for $y^2=x^3+2$.}
\end{table}	

The associated weight $k=2$ cusp form, and $L$-function are
\begin{equation}
f(s)=\sum_{n \geq 1}a_n q^n=q-q^7-5q^{13}+7q^{19}+ \cdots ,
\end{equation}

and
\begin{eqnarray}
L(s)&=&\sum_{n \geq 1}\frac{a_n}{n^s} \nonumber \\ &=& \prod_{p |N} \left ( 1-\frac{a_p}{p^s} \right )^{-1} \prod_{p \nmid N} \left ( 1-\frac{a_p}{p^s} +\frac{1}{p^{2s-1}}\right )^{-1} \\ 
&=& 1-\frac{1}{7^s}-\frac{5}{13^s}+\frac{7}{19^s}+ \cdots  \nonumber ,
\end{eqnarray}
where $q=e^{ i 2 \pi} $, respectively. The coefficients are generated using $a_p=p+1-\# E(\mathbb{F}_p)$, and the formulas
\begin{enumerate}
\item $a_{pq}=a_pa_q$  if $\gcd(p,q)=1$;
\item $a_{p^{n+1}}=a_{p^n}a_p-pa_{p^{n-1}}$  if $n \geq 2$.
\end{enumerate}

The corresponding functional equation is 
\begin{equation}
\Lambda(s)=\left ( \frac{\sqrt{N}}{2 \pi} \right ) ^s \Gamma(s) L(s)  \qquad \text{ and } \qquad \Lambda(s)=\Lambda(2-s),
\end{equation}
where $N=1728$, see \cite[p.\ 80]{KN93}.
}
\end{exa}

\begin{exa}   { \normalfont  The nonsingular elliptic curve $E: y^2=x^3+ 6x-2$ over the rational numbers has no complex multiplication and zero rank, it is listed as 1728.w1 in \cite{LMFDB}. The numerical data shows that $\# E(\mathbb{F}_p)=n$ is prime for at least one prime, see Table \ref{t804}. Hence, by Theorem \ref{thm800.1}, the corresponding group of $\mathbb{F}_p$-rational points $\# E(\mathbb{F}_p)$ has prime orders $n=\# E(\mathbb{F}_p)$ for infinitely many primes $p \geq 3$. \\

Since $\Delta=-2^6 \cdot 3^5$, the discriminant of the quadratic field $\mathbb{Q}(\sqrt{\Delta})$ is $D=-3$. Moreover, the integer divisor $d_E=1$ since $\# E(\mathbb{F}_p)$ is prime for at least one prime. Thus, applying Lemma \ref{lem800.11}, gives the natural density
\begin{equation}
\delta(1,E)=\frac{10}{9} P_0 \approx 0.5612957424882619712979385 \ldots,
\end{equation}
The predicted number of elliptic primes $p \nmid 6N$ such that $\# E(\mathbb{F}_p)$ is prime has the asymptotic 
formula
\begin{equation}
\pi(x,E)=\delta(1,E)\int_2^x \frac{1}{\log(t+1)} \frac{dt}{t}.
\end{equation}

A table for the prime counting function $\pi(1,E)$, for $2 \times 10^7 \leq x \leq 10^9$, and other information on 
this elliptic curve appears in \cite{ZD09}. \\

A lower bound for the counting function is
\begin{equation}
\pi(x,E)\geq \delta(1,E) \frac{x}{\log^ 3 x} \left (1+ O\left (\frac{1}{\log x} \right ) \right ),
\end{equation}
see Theorem \ref{thm800.1}.\\

\begin{table}
		\begin{center}
		\begin{tabular}{||l|r c l||} 
		\hline
		\textbf{Invariant} & \textbf{Value}&&   \\ [1ex]  
		\hline\hline
		Discriminant  &$\Delta$&=&$-16(4a^3+27b^2)=-2^6 \cdot 3^5$\\ 
		\hline
		Conductor& $N$&=&$2^6 \cdot 3^3  $\\
		\hline
		j-Invariant &$j(E)$&=&$ (-48b)^3/\Delta=2^9 \cdot 3$\\
		\hline
		Rank &$\rk(E)$&=& $0$ \\ 
		\hline
Special $L$-Value &$L(E,1)$&$\approx$&$ 2.24402797314$\\
	\hline
		Regulator &$R$&=&$ 1$ \\ 
		\hline
		Real Period &$\Omega$&=&$2.2440797314$ \\ 
		\hline
		Torsion Group &$E(\mathbb{Q})_{\text{tors}}$&=&$\{ \mathcal{O} \}$ \\ 
		\hline
		Integral Points &$E(\mathbb{Z})$&=&$ \{ \mathcal{O} \}$ \\
\hline
Rational Group &$E(\mathbb{Q})$&$=$&$ \{\mathcal{O}\}$ \\
		\hline
		Endomorphims Group &$End(E)$&=&$\mathbb{Z}$, nonCM \\
\hline
Integer Divisor  &$d_E$&=&$1$ \\
		\hline
		\end{tabular}
		\end{center}
\caption{\label{999} Data for $y^2=x^3+6x-2$.}
\end{table}

The associated weight $k=2$ cusp form, and $L$-function are
\begin{equation}
f(s)=\sum_{n \geq 1}a_n q^n=q+2q^5+q^7+2q^{11}-q^{13}-6q^{17}+5q^{19}+ \cdots ,
\end{equation}

and
\begin{eqnarray}
L(s)&=&\sum_{n \geq 1}\frac{a_n}{n^s} \nonumber \\ &=& \prod_{p |N} \left ( 1-\frac{a_p}{p^s} \right )^{-1} \prod_{p \nmid N} \left ( 1-\frac{a_p}{p^s} +\frac{1}{p^{2s-1}}\right )^{-1} \\ 
&=& 1+\frac{2}{5^s}+\frac{1}{7^s}+\frac{2}{11^s}-\frac{1}{13^s}-\frac{6}{17^s}+\frac{5}{19^s}+ \cdots  \nonumber ,
\end{eqnarray}
where $q=e^{ i 2 \pi} $, respectively. The coefficients are generated using $a_p=p+1-\# E(\mathbb{F}_p)$, and the formulas
\begin{enumerate}
\item $a_{pq}=a_pa_q$  if $\gcd(p,q)=1$;
\item $a_{p^{n+1}}=a_{p^n}a_p-pa_{p^{n-1}}$  if $n \geq 2$.
\end{enumerate}
The corresponding functional equation is 
\begin{equation}
\Lambda(s)=\left ( \frac{\sqrt{N}}{2 \pi} \right ) ^s \Gamma(s) L(s)  \qquad \text{ and } \qquad \Lambda(s)=\Lambda(2-s),
\end{equation}
where $N=1728$, see \cite[p.\ 80]{KN93}.
}
\end{exa}

\begin{exa}  { \normalfont The nonsingular elliptic curve $E: y^2=x^3-x$ over the rational numbers has complex multiplication by $\mathbb{Z}[i]$, and zero rank, it is listed as 32.a3  in \cite{LMFDB}. The numerical data shows that $\# E(\mathbb{F}_p)/8=n$ is prime for at least one prime, see (\ref{800-95}). Hence, by Theorem \ref{thm800.1}, the corresponding group of $\mathbb{F}_p$-rational points $\# E(\mathbb{F}_p)$ has prime orders $n=\# E(\mathbb{F}_p)$ for infinitely many primes $p \geq 3$. \\

Since $\Delta=-2^6$, the discriminant of the quadratic field $\mathbb{Q}(\sqrt{\Delta})$ is $D=-4$. Moreover, the integer divisor $d_E=8$ since $\# E(\mathbb{F}_p)/8$ is prime for at least one prime, see Table \ref{t808}. Thus, Lemma \ref{lem800.11} is not applicable. The natural density
\begin{equation}
\delta(8,E)=\frac{1}{2}\prod_{p\geq 3} \left (1 -\chi(p)\frac{p^2-p-1}{(p-\chi(p))(p-1)^2}\right )  \approx 0.5336675447 \ldots,
\end{equation}
where $\chi(n)=(-1)^{(n-1)/2}$, is computed in \cite[Lemma 7.1]{ZD09}. The predicted number of elliptic primes $p \nmid 6N$ such that $\# E(\mathbb{F}_p)$ is prime has the asymptotic 
formula
\begin{equation}
\pi(x,E)=\delta(8,E)\int_9^x \frac{1}{\log(t+1)- \log 8} \frac{dt}{\log t}.
\end{equation}

A table for the prime counting function $\pi(x,E)$, for $2 \times 10^7 \leq x \leq 10^9$, and other information on 
this elliptic curve appears in \cite{ZD09}. \\

A lower bound for the counting function is
\begin{equation}
\pi(x,E)\geq \delta(8,E) \frac{x}{\log^ 3 x} \left (1+ O\left (\frac{1}{\log x} \right ) \right ),
\end{equation}
see Theorem \ref{thm800.1}.\\

\begin{table}
\begin{center}
\begin{tabular}{||l|r c l||} 
		\hline
		\textbf{Invariant} & \textbf{Value}&&   \\ [1ex]  
		\hline\hline
		Discriminant  &$\Delta$&=&$-16(4a^3+27b^2)=-2^6$\\ 
		\hline
		Conductor& $N$&=&$2^5 $\\
		\hline
		j-Invariant &$j(E)$&=&$ (-48b)^3/\Delta=2^6 \cdot 3^3$\\
		\hline
		Rank &$\rk(E)$&=& $0$ \\ 
		\hline
Special $L$-Value &$L(E,1)$&$\approx$&$ .655514388573$\\
	\hline
		Regulator &$R$&=&$ 1$ \\ 
		\hline
		Real Period &$\Omega$&=&$5.24411510858$ \\ 
		\hline
		Torsion Group &$E(\mathbb{Q})_{\text{tors}}$&$\cong$&$\mathbb{Z}_2 \times \mathbb{Z}_2$ \\ 
		\hline
		Integral Points &$E(\mathbb{Z})$&=&$ \{ \mathcal{O} ,(\pm1,0);(0,0)\}$ \\
\hline
Rational Group &$E(\mathbb{Q})$&$=$&$ E(\mathbb{Q})_{\text{tors}}$ \\		
		\hline
		Endomorphims Group &$End(E)$&=&$\mathbb{Z}[i]$, CM \\
\hline
Integer Divisor  &$d_E$&=&$2,4,8$ \\
		\hline
\end{tabular}
\end{center}
\caption{\label{959} Data for $y^2=x^3-x$.}
\end{table}	

The associated weight $k=2$ cusp form, and $L$-function are
\begin{equation}
f(s)=\sum_{n \geq 1}a_n q^n=q+2q^5+q^7+2q^{11}-q^{13}-6q^{17}+5q^{19}+ \cdots ,
\end{equation}

and
\begin{eqnarray}
L(s)&=&\sum_{n \geq 1}\frac{a_n}{n^s} \nonumber \\ &=& \prod_{p |N} \left ( 1-\frac{a_p}{p^s} \right )^{-1} \prod_{p \nmid N} \left ( 1-\frac{a_p}{p^s} +\frac{1}{p^{2s-1}}\right )^{-1} \\ 
&=& 1+\frac{2}{5^s}+\frac{1}{7^s}+\frac{2}{11^s}-\frac{1}{13^s}-\frac{6}{17^s}+\frac{5}{19^s}+ \cdots  \nonumber ,
\end{eqnarray}
where $q=e^{ i 2 \pi} $, respectively. The coefficients are generated using $a_p=p+1-\# E(\mathbb{F}_p)$, and the formulas
\begin{enumerate}
\item $a_{pq}=a_pa_q$  if $\gcd(p,q)=1$;
\item $a_{p^{n+1}}=a_{p^n}a_p-pa_{p^{n-1}}$  if $n \geq 2$.
\end{enumerate}

The corresponding functional equation is 
\begin{equation}
\Lambda(s)=\left ( \frac{\sqrt{N}}{2 \pi} \right ) ^s \Gamma(s) L(s)  \qquad \text{ and } \qquad \Lambda(s)=\Lambda(2-s),
\end{equation}
where $N=1728$, see \cite[p.\ 80]{KN93}.
}
\end{exa}

\section{Exercises}

\textbf{Problem 1.} Assume random elliptic curve $E:f(x,y)=0$ has no CM, $P$ is a point of infinite order. Is the integer divisor $d_E < \infty$ bounded? This parameter is bounded for CM elliptic curves, in fact $d_E <24$, reference: \cite[Proposition 1]{JJ08}.\\

\textbf{Problem 2.} Assume the elliptic curve $E:f(x,y)=0$ has no CM, $P$ is a point of infinite order, and the density $\delta(1,E)>0$. What is the least prime $p\nmid 6N$ such that the integer divisor $d_E=1$? Reference: \cite[Corollary 1.2]{CA03}.\\

\textbf{Problem 3.} Fix an  elliptic curve $E:y^2=x^3+ax+b$ of rank $\rk(E)>0$, and CM; and $P$ is a point of infinite order. Let the integer divisor $d_E=4$. Assume the densities $\delta(1,E)>0$, $\delta(2,E)>0$, and $\delta(4,E)>0$ are defined. What is the arithmetic relationship between the densities?\\

\textbf{Problem 4.} Fix an  elliptic curve $E:y^2=x^3+ax+b$ of rank $\rk(E)>0$, and CM; and $P$ is a point of infinite order. Does $\#E(\mathbb{F}_p)=n$ prime for at least one large prime $p \geq 2$ implies that the integer divisor $d_E=1$?\\


\newpage 
\begin{thebibliography}{999}
\bibitem{BC11} Balog, Antal; Cojocaru, Alina-Carmen; David, Chantal. \textit{\color{blue}Average twin prime conjecture for elliptic curves}. Amer.
		J. Math. 133,(2011), no. 5, 1179-1229.

\bibitem{BP75} Borosh, I.; Moreno, C. J.; Porta, H. \textit{\color{blue}Elliptic curves over finite fields. II}. Math. Comput. 29 (1975), 951-964. 

\bibitem{CA05} Cojocaru, Alina Carmen. \textit{\color{blue}Reductions of an elliptic curve with almost prime orders}. Acta Arith. 119 (2005), no. 3, 265-289. 

\bibitem{CA04} Cojocaru, Alina Carmen. \textit{\color{blue}Questions about the reductions modulo primes of an elliptic curve}.  Number theory,  61-79, CRM Proc. Lecture Notes, 36, Amer. Math. Soc., Providence, RI, 2004. 

\bibitem{CA03} Cojocaru, Alina Carmen. \textit{\color{blue}Cyclicity of CM elliptic curves modulo p}. Trans. Amer. Math. Soc. 355 (2003), no. 7, 2651-2662.

\bibitem{DL12} De Koninck, Jean-Marie; Luca, Florian. \textit{\color{blue}Analytic number theory. Exploring the anatomy of integers}. Graduate Studies in Mathematics, 134. American Mathematical Society, Providence, RI, 2012.

\bibitem{FI10} Friedlander, John; Iwaniec, Henryk. \textit{\color{blue}Opera de cribro}. American Mathematical Society Colloquium Publications, 57. American Mathematical Society, Providence, RI, 2010.

\bibitem{GM86} Gupta, Rajiv; Murty, M. Ram. \textit{\color{blue}Primitive points on elliptic curves}. Compositio Math.  58  (1986),  no. 1, 13-44.
 
\bibitem{GM90} Gupta, Rajiv; Murty, M. Ram. \textit{\color{blue}Cyclicity and generation of points mod p on elliptic curves}. Invent. Math. 101 (1990), no. 1, 225-235. 
		
\bibitem{HG07} Harman, Glyn. \textit{\color{blue}Prime-detecting sieves}. London Mathematical Society Monographs Series, 33. Princeton University Press, Princeton, NJ, 2007. 

\bibitem{IJ08} Iwaniec, Henryk; Jimenez Urroz, Jorge. \textit{\color{blue}Orders of CM elliptic curves modulo p with at most two primes}. Ann. Sc. Norm. Super. Pisa Cl. Sci. (5) 9 (2010), no. 4, 815-832.	
		
\bibitem{IK04} Iwaniec, Henryk; Kowalski, Emmanuel. \textit{\color{blue}Analytic number theory}. AMS Colloquium Publications,
		53. American Mathematical Society, Providence, RI, 2004.

\bibitem{JN09} Jones, Nathan. \textit{\color{blue}Averages of elliptic curve constants}. Math. Ann. 345 (2009), no. 3, 685-710. 

\bibitem{JJ08} Jimenez Urroz, Jorge. \textit{\color{blue}Almost prime orders of CM elliptic curves modulo p}. Algorithmic number theory, 74-87, Lecture Notes in Comput. Sci., 5011, Springer, Berlin, 2008.

\bibitem{JN10} Jones, Nathan. \textit{\color{blue}Almost all elliptic curves are Serre curves}. Trans. Amer. Math. Soc. 362 (2010), no. 3, 1547-1570.

\bibitem{KA93} Anatolij A. Karacuba. \textit{\color{blue}Basic Analytic Number Theory}, Springer-Verlag, 1993.

\bibitem{KN88} Koblitz, Neal. \textit{\color{blue}Primality of the number of points on an elliptic curve over a finite field}. Pacific J. Math.  131  (1988),  no. 1, 157-165.
		
\bibitem{KN93} Koblitz, Neal. \textit{\color{blue}Introduction to elliptic curves and modular forms}. Second edition. Graduate Texts in Mathematics, 97. Springer-Verlag, New York, 1993.

\bibitem{LT77} Lang, S.; Trotter, H. \textit{\color{blue}Primitive points on elliptic curves}. Bull. Amer. Math. Soc. 83 (1977), no. 2, 289-292.

\bibitem{LMFDB} The L-functions and modular forms database, www.lmfdb.org.		
			
\bibitem{LS14} Lenstra, H. W., Jr.; Stevenhagen, P.; Moree, P. \textit{\color{blue}Character sums for primitive root densities}. Math. Proc. Cambridge Philos. Soc. 157 (2014), no. 3, 489-511. And arXiv:1112.4816.

\bibitem{MG15} Giulio Meleleo. \textit{\color{blue}Questions Related To Primitive Points On Elliptic Curves And Statistics For Biquadratic Curves Over
Finite Fields}, Thesis, Universita degli Studi, Roma Tre, 2015.
				
\bibitem{MV07} Montgomery, Hugh L.; Vaughan, Robert C. \textit{\color{blue}Multiplicative number theory. I. Classical theory}. Cambridge University Press, Cambridge, 2007.

\bibitem{SA43}  Selberg, Atle. \textit{\color{blue}On the normal density of primes in small intervals, and the difference between consecutive primes.} Arch. Math. Naturvid. 47, (1943). no. 6, 87-105.

\bibitem{SJ09} Silverman, Joseph H. \textit{\color{blue}The arithmetic of elliptic curves}. Second edition. Graduate Texts in Mathematics, 106. Springer, Dordrecht, 2009.

\bibitem{SJ05} Silverman, Joseph H. \textit{\color{blue}The p-adic properties of division polynomials and elliptic divisibility sequences}. Math. Ann.  332  (2005),  no. 2, 443-471.
		
\bibitem{SV11} Shparlinski, Igor E.; Voloch, Jose Felipe. \textit{\color{blue}Generators of elliptic curves over finite fields}. Preprint 2011.

\bibitem{SZ03} Schmitt, Susanne; Zimmer, Horst G. \textit{\color{blue}Elliptic curves. A computational approach}. With an appendix by Attila Petho. de Gruyter 
Studies in Mathematics, 31. Walter de Gruyter $\&$ Co., Berlin, 2003. 

\bibitem{TG15} Tenenbaum, Gerald. \textit{\color{blue}Introduction to analytic and probabilistic number theory}. Translated from the Third French edition. American Mathematical Society, Rhode Island, 2015.

\bibitem{ZD09} Zywina, David. \textit{\color{blue}A refinement of Koblitz's conjecture}, arXiv:0909.5280.

\bibitem{ZD08} Zywina, David. \textit{\color{blue}The Large Sieve and Galois Representations}, arXiv:0812.222.

\bibitem{WN95} Watt, N. \textit{\color{blue}Short intervals almost all containing primes}. Acta Arith. 72 (1995), no. 2, 131-167.





\end{thebibliography}
\end{document}